\documentclass[letterpaper,10pt,conference]{ieeeconf} 

\IEEEoverridecommandlockouts
\usepackage{times,amsmath,epsfig,amssymb,graphicx,amsfonts}
\usepackage{empheq}
\usepackage{graphicx}
\usepackage{psfrag}
\usepackage{fge}
\usepackage{amsmath}
\usepackage{setspace}
\usepackage{color}
\usepackage{amsfonts}   
\usepackage{amssymb}    
\usepackage{mathrsfs}
\usepackage{setspace}
\usepackage{dblfloatfix}
\usepackage{tikz}
\usepackage{tkz-tab}
\usepackage{algorithm}
\usepackage{algcompatible}
\usepackage{lipsum}
\usepackage{color}
\usepackage{mathrsfs}
\usepackage{epstopdf}
\usepackage{enumerate}
\usepackage{url}
\usepackage{cite}

\floatname{algorithm}{Algorithm}

\floatname{algorithm}{Algorithm}

\newtheorem{mydef}{Definition}
\newtheorem{myassump}{Assumption}
\newtheorem{mytheorem}{Theorem}
\newtheorem{mylemma}{Lemma}
\newtheorem{myremark}{Remark}
\newtheorem{mycorollary}{Corollary}
\newtheorem{myproposition}{Proposition}

\newcounter{ale}

\newenvironment{liste}{\begin{itemize}}{\end{itemize}}
\newcommand{\aliste}{\begin{liste} \setcounter{ale}{1}}
\newcommand{\zliste}{\end{liste}}

\title{{\LARGE {\bf Near-Optimal Sensor Scheduling for Batch State Estimation:\\
Complexity, Algorithms, and Limits}}}
\author{Vasileios Tzoumas{$^{1}$}, Ali Jadbabaie{$^{2}$}, George J.~Pappas{$^{1}$}
\thanks{$^{1}$The authors are with the Department of Electrical and Systems Engineering, University of Pennsylvania, Philadelphia, PA 19104-6228 USA (email: {\fontsize{8}{8}\selectfont\ttfamily\upshape \{vtzoumas, pappasg\}@seas.upenn.edu}).}
\thanks{$^{2}$The author is with the Department of Civil and Environmental Engineering, Massachusetts Institute of Technology, Cambridge, MA, 02139 USA (email: {\fontsize{8}{8}\selectfont\ttfamily\upshape jadbabai@mit.edu}).}
\thanks{This work was supported in part by TerraSwarm, one of six centers of STARnet, a Semiconductor Research Corporation program sponsored by MARCO and DARPA, in part by AFOSR Complex Networks Program and in part by AFOSR MURI CHASE.}
}

\begin{document}
\maketitle

\begin{abstract}
In this paper, we focus on batch state estimation for linear systems.  This problem is important in applications such as environmental field estimation, robotic navigation, and target tracking.  Its difficulty lies on that limited operational resources among the sensors, e.g., shared communication bandwidth or battery power, constrain the number of sensors that can be active at each measurement step.  As a result, sensor scheduling algorithms must be employed.  Notwithstanding, current sensor scheduling algorithms for batch state estimation scale poorly with the system size and the time horizon.  In addition, current sensor scheduling algorithms for Kalman filtering, although they scale better, provide no performance guarantees or approximation bounds for the minimization of the batch state estimation error.  In this paper, one of our main contributions is to provide an algorithm that enjoys both the estimation accuracy of the batch state scheduling algorithms and the low time complexity of the Kalman filtering scheduling algorithms.  In particular: 1) our algorithm is near-optimal: it achieves a solution up to a multiplicative factor $1/2$ from the optimal solution, and this factor is close to the best approximation factor $1/e$ one can achieve in polynomial time for this problem; 2) our algorithm has (polynomial) time complexity that is not only lower than that of the current algorithms for batch state estimation; it is also lower than, or similar to, that of the current algorithms for Kalman filtering.  We achieve these results by proving two properties for our batch state estimation error metric, which quantifies the square error of the minimum variance linear estimator of the batch state vector: a) it is supermodular in the choice of the sensors; b) it has a sparsity pattern (it involves matrices that are block tri-diagonal) that facilitates its evaluation at each sensor~set.  
\end{abstract}

\section{Introduction}\label{sec:Intro}

Search and rescue \cite{kumar2004robot}, environmental field estimation \cite{nowak2004estimating}, robotic navigation \cite{vitus2011sensor}, and target tracking \cite{masazade2012sparsity} are only a few of the challenging information gathering problems that employ the monitor capabilities of sensor networks~\cite{rowaihy2007survey}.
In particular, all these problems face the following three main~challenges:
\begin{itemize}
\item they involve systems whose evolution is largely unknown, corrupted with noisy inputs \cite{masazade2012sparsity}, and sensors with limited sensor capabilities, corrupted with measurement noise \cite{kailath2000linear}.
\item they involve systems that change over time \cite{nowak2004estimating}, and as a result, necessitate both spacial and temporal deployment of sensors in the environment. At the same~time:
\item they involve operational constraints, such as limited bandwidth and battery life, which limit the number of sensors that can be simultaneously used (i.e., be switched-on) in the information gathering process \cite{hero2011sensor}.
\end{itemize}

As a result of these challenges, researchers focused on the following question: ``How do we select at each measurement step only a few sensors so to minimize the estimation error despite the above challenges?''  The effort to answer this question resulted to the problem of sensor scheduling \cite{hero2011sensor}: in particular, sensor scheduling offers a formal methodology to use at each measurement time only a few sensors and obtain an optimal
trade-off between the estimation accuracy and the usage of the limited operational resource (e.g., the shared bandwidth).
Clearly, sensor scheduling is a combinatorial problem of exponential complexity \cite{rowaihy2007survey}. 

In this paper, we focus on the following instance of this problem: 

\paragraph*{Problem 1 (Sensor Scheduling for Minimum Variance Batch State Estimation)} 

\textit{Consider a time-invariant linear system, whose state at time $t_k$ is denoted as $x(t_k)$, a set of $m$ sensors, and a fixed set of $K$ measurement times $t_1,t_2,\ldots, t_K$.  In addition, consider that at each $t_k$ at most $r_k$ sensors can be used, where $r_k \leq m$.  At each $t_k$ select a set of $r_k$ sensors so to minimize the square estimation error of the minimum variance linear estimator of the batch state vector $(x(t_1),x(t_2), \ldots, x(t_K))$.}

There are two classes of sensor scheduling algorithms, that trade-off between the estimation accuracy of the batch state vector and their time complexity: these for Kalman filtering, and those for batch state estimation.  In more detail:

\paragraph*{Kalman filtering algorithms} These algorithms sacrifice estimation accuracy over reduced time complexity.  The reason is that they are sequential algorithms: at each $t_k$, they select the sensors so to minimize the square estimation error of the minimum variance linear estimator of $x(t_k)$ (given the measurements up to~$t_k$).  Therefore, their objective is to minimize the sum of the square estimation errors of $x(t_k)$ across the measurement times $t_k$ \cite{liu2014optimal}. However, this sum is only an upper bound to the square estimation error of the batch state vector $(x(t_1),x(t_2), \ldots, x(t_K))$.  Thus, the Kalman filtering algorithms lack on estimation accuracy with respect to the batch state estimation algorithms. 

\paragraph*{Batch state estimation algorithms} These algorithms sacrifice time complexity over estimation accuracy.  The reason is that they perform global optimization, in accordance to Problem 1.  Therefore, however, they lack on time complexity with respect to the Kalman filtering algorithms.  

Notwithstanding, in several recent robotic applications, batch estimation algorithms have been proven competitive in their time complexity to their filtering counterparts~\cite{kaess2008isam,anderson2015batch}. The reason is that sparsity patterns emerge in these applications, that reduce the time complexity of their batch estimation algorithms to an order similar to that of the filtering algorithms~\cite{strasdat2010real}. 
Thereby, the following question on Problem 1 arises: 

\paragraph*{Question 1} \textit{``Is there an algorithm for Problem 1 that enjoys both the estimation accuracy of the batch state algorithms and the low time complexity of the Kalman filtering algorithms?''}  

\paragraph*{Literature review on sensor scheduling algorithms for batch state estimation}  The most relevant paper on Problem 1 is \cite{Roy2016369}, where an algorithm based on convex relaxation is provided.  This algorithm scales poorly with the system's size and number of measurement times.  In addition, it provides no approximation performance guarantees.

\paragraph*{Literature review on sensor scheduling algorithms for Kalman filtering} 
Several papers in this category have focused on myopic algorithms \cite{shamaiah2010greedy}; such algorithms, however, often perform poorly \cite{liu2003multi}.  Other papers have focused on algorithms that use: tree pruning~\cite{vitus2012efficient}, convex optimization~\cite{joshi2009sensor}, quadratic programming~\cite{mo2011sensor}, or submodular function maximization \cite{zhang2015Sensor,jawaid2015submodularity}.  Nevertheless, these algorithms provide no performance guarantees on the batch state estimation error, or have time complexity that scales poorly with the system's size and number of measurement times~\cite{vitus2012efficient}~\cite{mo2011sensor}.  To reduce the time complexity of these algorithms, papers have also proposed periodic sensor schedules~\cite{liu2014optimal}. 

\paragraph*{Contributions} We now present our contributions:
\begin{enumerate}[1)]
\item We prove that Problem 1 is NP-hard.
\item We provide an algorithm for Problem 1 (Algorithm 1) that answers Question 1 positively. The reasons are two:
\begin{enumerate}[i)]
\item Algorithm 1 is near-optimal: it achieves a solution that is up to a multiplicative factor $1/2$ from the optimal solution.  In addition, this multiplicative factor is close to the factor $1/e$ which we prove to be the best approximation factor one can achieve in polynomial time for Problem 1 in the worst-case. 
 
\item Algorithm 1 has (polynomial) time complexity that is not only lower than that of the state of the art scheduling algorithms for batch state estimation; it is also lower than, or similar to, that of the state of the art scheduling algorithms for Kalman filtering. For example, it has similar complexity to the state of the art periodic scheduling algorithm in \cite{liu2014optimal} (in particular: lower for $K$ large enough), and lower than the complexity of the algorithm in~\cite{joshi2009sensor}. 
\end{enumerate}

Overall, in response to Question 1, Algorithm 1 enjoys both the higher estimation accuracy of the batch state estimation approach (compared to the Kalman filtering approach, that only approximates the batch state estimation error with an upper bound) and the low time complexity of Kalman filtering approach.

\item We prove limits on the minimization of the square error of the minimum variance estimator of $(x(t_1),$ $x(t_2), \ldots, x(t_K))$ with respect to the scheduled sensors.  For example, we prove that the number $r_k$ of used sensors at each measurement time must increase linearly with the system size for fixed estimation error and number of measurement times $K$; this is a fundamental limit, especially for large-scale systems. 
\end{enumerate}

\paragraph*{Our technical contributions} We achieve our aforementioned contributions by proving the following two:

\paragraph{Supermodularity in Problem 1} We prove that our estimation metric, that quantifies the square error of the minimum variance estimator of $(x(t_1),$ $x(t_2), \ldots, x(t_K))$, is a supermodular function in the choice of the used sensors.  This result becomes important when we compare it to results on the multi-step Kalman filtering that show that the corresponding estimation metric in this case is neither supermodular nor submodular~\cite{jawaid2015submodularity,zhang2015Sensor}.\footnote{The observation of \cite{jawaid2015submodularity} is also important as it disproves previous results in the literature~\cite{huber2009distributed}.}

In addition, this submodularity result cannot be reduced to the batch estimation problem in \cite{krause2008near}.  The main reasons are two: i) we consider sensors that can measure any linear combination of the element of $x(t_k)$, in contrast to \cite{krause2008near}, where each sensor measures directly only one element of $x(t_k)$.  Nonetheless, the latter assumption is usually infeasible in dynamical systems \cite{Chen:1998:LST:521603}; ii) our error metric is relevant to estimation problems for dynamical systems and different to the submodular information gain considered in \cite{krause2008near}.

\paragraph{Sparsity in Problem 1} We identify a sparsity pattern in our error metric, that facilitates the latter's evaluation at each sensor set.  In particular, we prove that the error covariance of the minimum variance linear estimator of the batch state vector is block tri-diagonal. 

We organize the rest of the paper as follows:  In Section \ref{sec:Prelim} we present formally Problem 1.  In Section \ref{sec:main}, we present in three subsections our main results: in Section \ref{sec:comp_complx}, we prove that our sensor scheduling problem is NP-hard.  In Section \ref{sec:alg}, we derive a near-optimal approximation algorithm, and compare its time complexity with that of previous sensor scheduling algorithms.  In Section \ref{sec:limitations}, we prove limits on the minimization of the batch state estimation error with respect to the used sensors.  Section \ref{sec:conc} concludes the paper with our future work. \emph{All proofs are found in the Appendix.}\footnote{\emph{Standard notation is presented in this footnote:} We denote the set of natural numbers $\{1,2,\ldots\}$ as $\mathbb{N}$, the set of real numbers as  $\mathbb{R}$,  and the set $\{1, 2, \ldots, n\}$ as $[n]$ ($n \in \mathbb{N}$).  The empty set is denoted as $\emptyset$.  Given a set $\mathcal{X}$, $|\mathcal{X}|$ is its cardinality.  
Matrices are represented by capital letters and vectors by lower-case letters.  We write $A \in \mathcal{X}^{n_1 \times n_2}$ ($n_1, n_2 \in \mathbb{N}$) to denote a matrix of $n_1$ rows and $n_2$ columns whose elements take values in $\mathcal{X}$.  Moreover, for a matrix $A$, $A^\top$ is its transpose, and $[A]_{ij}$ is its element at the $i$-th row and $j$-th column.
In addition, $\|A\|_2\equiv \sqrt{A^\top A}$ is its spectral norm, and $\det(A)$ its determinant.  Furthermore, if $A$ is positive semi-definite or positive definite, we write $A \succeq 0$ and $A\succ {0}$, respectively.  
$I$ is the identity matrix; its dimension is inferred from the context.  Similarly for the zero matrix $0$.
Finally, for a random variable $x \in \mathbb{R}^n$, $\mathbb{E}(x)$ is its expected value, and $\mathbb{C}(x)$ its covariance.}

\section{Problem Formulation} \label{sec:Prelim}

In the following paragraphs, we present our sensor scheduling problem for batch state estimation.
To this end, we first build our system and measurement framework.  Then, we define our sensor scheduling framework and, finally, present our sensor scheduling problem.  

We start in more detail with the system model:

\paragraph*{System Model}
\textit{We consider the linear time-invariant system:
\begin{equation}\label{eq:dynamics}
\dot{x}(t)= Ax(t)+Bu(t)+Fw(t), t \geq t_0,
\end{equation}
where $t_0$ is the initial time, $x(t) \in \mathbb{R}^n$ ($n \in \mathbb{N}$) the state vector, $\dot{x}(t)$ the time derivative of $x(t)$, $u(t)$ the exogenous input, and $w(t)$ the process noise. The system matrices $A,B$ and $F$ are of appropriate dimensions.  We consider that $u(t)$, $A,B$ and $F$ are known.  Our main assumption on $w(t)$ is found in Assumption \ref{ass:indep}, that is presented after our measurement model.}

\begin{myremark}
Our results extend to continuous and discrete time-variant systems, as explained in detail in Section \ref{sec:main} (Corollaries \ref{cor:cont} and \ref{cor:discrete}).
\end{myremark}

We introduce the measurement model: 

\paragraph*{Measurement Model}
\textit{We consider $m$ sensors:
\begin{equation}\label{eq:sensors}
z_i(t)= C_ix(t)+v_i(t), i \in [m],
\end{equation}
where $z_i(t)$ is the measurement taken by sensor $i$ at time $t$, $C_i \in \mathbb{R}^{d_i\times n}$ ($d_i \in \mathbb{N}$) is sensor's $i$ measurement matrix, and $v_i(t)$ is its measurement noise.}


We make the following assumption on $x(t_0)$, $w(t)$ and~$v_i(t)$:

\begin{myassump}\label{ass:indep}
\textit{For all $t$, $t'\geq t_0$, $t\neq t'$, and all $i \in [m]$: $x(t_0)$, $w(t)$, $w(t')$, $v_i(t)$ and $v_i(t')$ are uncorrelated; in addition, $x(t_0)$, $w(t)$ and $v_i(t)$ have positive definite covariance.}
\end{myassump}


We now introduce the sensor scheduling model:

\paragraph*{Sensor Scheduling Model}
\textit{The $m$ sensors  in \eqref{eq:sensors} are used at $K$ scheduled measurement times $\{t_1,t_2,\ldots, t_K\}$.  Specifically,  at each $t_k$ only $r_k$ of these $m$ sensors are used ($r_k \leq m$), resulting in the batch measurement vector $y(t_k)$:
\begin{equation}\label{eq:measurements}
y(t_k)= S(t_k)z(t_k), k \in [K],
\end{equation}
where $z(t_k)\equiv (z_1^\top(t_k),z_2^\top(t_k),\ldots,z_m^\top(t_k))^\top$, and $S(t_k)$ is the sensor selection matrix: it is a block matrix, composed of matrices $[S(t_k)]_{ij}$ ($i \in [r_k]$, $j \in [m]$) such that $[S(t_k)]_{ij}=I$ if sensor $j$ is used at $t_k$, and $[S(t_k)]_{ij}=0$ otherwise.  We consider that each sensor can be used at most once at each $t_k$, and as a result, for each $i$ there is one $j$ such that $[S(t_k)]_{ij}=I$ while for each $j$ there is at most one $i$ such that $[S(t_k)]_{ij}=I$.}

We now present the sensor scheduling problem we study in this paper.  To this end, we use two notations: 
\paragraph*{Notation} First, we set $\mathcal{S}_k\equiv \{j :  \text{there exists } i \in [r_k], [S(t_k)]_{ij}=1\}$; that is, $\mathcal{S}_k$ is the set of indices that correspond to used sensors at $t_k$. Second, we set  $\mathcal{S}_{1:K}\equiv (\mathcal{S}_1, \mathcal{S}_2,\ldots, \mathcal{S}_K)$.

\paragraph*{Problem 1 (Sensor Scheduling for Minimum Variance Batch State Estimation)} 
\textit{Given a set of measurement times $t_1, t_2, \ldots, t_K$, select at each $t_k$ to use a subset of $r_k$ sensors, out of the $m$ sensors in \eqref{eq:sensors}, so to minimize the $\log\det$ of the error covariance of the minimum variance linear estimator of $x_{1:K}\equiv(x(t_1),x(t_2), \ldots, x(t_K))$.  In mathematical notation:
\begin{equation*}\label{pr:entropy}
\begin{aligned}
& \underset{\mathcal{S}_k \subseteq [m], k \in [K]}{\text{minimize}} 
 \;  \log\det(\Sigma(\hat{x}_{1:K}|\mathcal{S}_{1:K})) \\
&\hspace{2.5mm}\text{subject to} 
\quad \hspace{0mm} |\mathcal{S}_k| \leq r_k, k \in [K],
\end{aligned}
\end{equation*}
where $\hat{x}_{1:K}$ is the minimum variance linear estimator of $x_{1:K}$, and $\Sigma(\hat{x}_{1:K}|\mathcal{S}_{1:K})$ its error covariance given $\mathcal{S}_{1:K}$.
}

Two remarks follow on the definition of Problem 1.  In the first remark we explain why we focus on $\hat{x}_{1:K}$, and in the second why we focus on $\log\det(\Sigma(\hat{x}_{1:K}))$.

\paragraph*{Notation} For notational simplicity, we use  $\Sigma(\hat{x}_{1:K})$ and $\Sigma(\hat{x}_{1:K}|\mathcal{S}_{1:K})$ interchangeably. 

\begin{myremark}\label{rem:mve}
We focus on the minimum variance linear estimator $\hat{x}_{1:K}$ because of its optimality: it minimizes among all linear estimators of $x_{1:K}$ the estimation error $\mathbb{E}(\|x_{1:K}-\hat{x}_{1:K}\|_2^2)$, where the expectation is taken with respect to $y(t_1), y(t_2), \ldots, y(t_K)$ \cite{kailath2000linear}.  Because $\hat{x}_{1:K}$ is also unbiased (that is, $\mathbb{E}(\hat{x}_{1:K})=x_{1:K}$, where the expectation is taken with respect to $y(t_1), y(t_2), \ldots, y(t_K))$, \emph{we equivalently say that $\hat{x}_{1:K}$ is the minimum variance estimator of $x_{1:K}$}.

We compute the error covariance of $\hat{x}_{1:K}$ in Appendix \ref{app:estimator}.
\end{myremark}

\begin{myremark}\label{rem:logdet}
{We focus on the estimation error metric $\log\det(\Sigma(\hat{x}_{1:K}))$ because when it is minimized the probability that the estimation error $\|x_{1:K}-\hat{x}_{1:K}\|_2^2$ is small is maximized.  To quantify this statement, we note that this error metric is related to the $\eta$-confidence ellipsoid of $x_{1:K}-\hat{x}_{1:K}$~\cite{joshi2009sensor}:  Specifically, the $\eta$-confidence ellipsoid is the minimum volume ellipsoid that contains $x_{1:K}-\hat{x}_{1:K}$ with probability $\eta$, that is, it is the $\mathcal{E}_\epsilon\equiv \{x : x^\top \Sigma(\hat{x}_{1:K}) x \leq \epsilon\}$,
where $\epsilon$ is the quantity $F^{-1}_{\chi^2_{n(k+1)}}(\eta)$, and $F_{\chi^2_{n(k+1)}}$ the cumulative distribution function of a $\chi$-squared random variable with $n(k+1)$ degrees of freedom~\cite{venkatesh2012theory}.  Thus, its volume
\begin{equation}\label{eq:ellipsoid}
\text{vol}(\mathcal{E}_\epsilon)\equiv \frac{(\epsilon\pi)^{n(k+1)/2}}{\Gamma\left(n(k+1)/2+1\right)}\det\left(\Sigma(\hat{x}_{1:K})^{1/2}\right),
\end{equation}
where $\Gamma(\cdot)$ denotes the Gamma function~\cite{venkatesh2012theory}, quantifies the estimation error of the optimal estimator $\hat{x}_{1:K}$.  Therefore, by taking the logarithm of \eqref{eq:ellipsoid}, we validate that when the $\log\det(\Sigma(\hat{x}_{1:K}))$ is minimized the probability that the estimation error $\|x_{1:K}-\hat{x}_{1:K}\|_2^2$ is small is maximized.}
\end{myremark}

\section{Main Results} \label{sec:main}

Our main results are presented in three sections:
\begin{itemize}
\item In Section \ref{sec:comp_complx}, we prove that Problem 1 is NP-hard.
\item In Section \ref{sec:alg}, we derive a provably near-optimal approximation algorithm for Problem 1. In addition, we emphasize on its time complexity and compare it to that of existing sensor scheduling algorithms for two categories: batch state estimation, and Kalman filtering.
\item In Section \ref{sec:limitations}, we prove limits on the optimization of the estimation error $\mathbb{E}(\|x_{1:K}-\hat{x}_{1:K}\|_2^2)$ with respect to the scheduled sensors. 
\end{itemize}

\subsection{Computational Complexity of Sensor Scheduling for Batch State Estimation}\label{sec:comp_complx}

In this section, we characterize the computational complexity of Problem 1. In particular, we prove: 

\begin{mytheorem}\label{th:np}
\textit{The problem of sensor scheduling for minimum variance batch state estimation (Problem 1) is NP-hard.}
\end{mytheorem}
\begin{proof}
The proof is found in Appendix \ref{app:np}.  Here, we note that the proof is complete by finding an instance of Problem 1 that is equivalent to the NP-hard minimal observability problem introduced in~\cite{olshevsky2014minimal,sergio2015minimal}.~
\end{proof}

Due to Theorem \ref{th:np}, for the polynomial time solution of Problem 1 we need to appeal to approximation algorithms.  To this end, in Section \ref{sec:alg}, we provide an efficient provably near-optimal approximation algorithm: 

\subsection{Algorithm for Sensor Scheduling for Minimum Variance Batch State Estimation}\label{sec:alg}

We propose Algorithm \ref{alg:general} for Problem 1 (Algorithm \ref{alg:general} uses Algorithm \ref{alg:greedy_alg} as a subroutine); with the following theorem, we quantify its approximation performance and time complexity.  

\begin{algorithm}[tl]
\caption{Approximation algorithm for Problem 1.}
\begin{algorithmic}
\REQUIRE  Number of measurement times $K$, scheduling constraints $r_1, r_2, \ldots, r_K$, estimation error function $\log\det(\Sigma(\hat{x}_{1:K}|\mathcal{S}_{1:K})): \mathcal{S}_k \subseteq [m], k \in [K] \mapsto \mathbb{R}$
\ENSURE Sensor sets $(\mathcal{S}_1, \mathcal{S}_2,\ldots, \mathcal{S}_K)$ that approximate the solution to Problem 1, as quantified in Theorem \ref{th:alg_performance}
\STATE $k \leftarrow 1$, $\mathcal{S}_{1:0} \leftarrow \emptyset$
\WHILE {$k\leq K$} \STATE{ 
\vspace{-4.9mm}
	\begin{enumerate}[1.]
    \item Apply Algorithm \ref{alg:greedy_alg} to
    \begin{equation}\label{eq:local_opt}
    \min_{S \subseteq [m]}\{\log\det(\Sigma(\hat{x}_{1:K}|\mathcal{S}_{1:k-1},\mathcal{S})): |\mathcal{S}|\leq r_k\}
    \end{equation}
    \item Denote as $\mathcal{S}_k$ the solution Algorithm \ref{alg:greedy_alg} returns
	\item $\mathcal{S}_{1:k} \leftarrow (\mathcal{S}_{1:k-1},\mathcal{S}_k)$
	\item $k \leftarrow k+1$
	\end{enumerate}	
	\vspace{-1.5mm}
	}
\ENDWHILE
\end{algorithmic} \label{alg:general}
\end{algorithm}

\begin{mytheorem}\label{th:alg_performance}
\textit{The theorem has two parts:}
\begin{enumerate}[1)]
\item \textit{\emph{Approximation performance of Algorithm \ref{alg:general}:} Algorithm \ref{alg:general} returns sensors sets $\mathcal{S}_1, \mathcal{S}_2, \ldots, \mathcal{S}_K$ that:  
\begin{itemize}
\item satisfy all the feasibility constraints of Problem 1: $|\mathcal{S}_k|\leq r_k, k \in [K]$
\item achieve an error value $\log\det(\Sigma(\hat{x}_{1:K}|\mathcal{S}_{1:K}))$, where $\mathcal{S}_{1:K}\equiv(\mathcal{S}_1,\mathcal{S}_2, \ldots, \mathcal{S}_K)$, such that:
\begin{equation}
\frac{\log\det(\Sigma(\hat{x}_{1:K}|\mathcal{S}_{1:K}))-OPT}{MAX-OPT}\leq \frac{1}{2}, \label{ineq:opt_guar}
\end{equation}
\end{itemize}
where $OPT$ is the (optimal) value to Problem 1, and $MAX$ is the maximum (worst) value to Problem 1  $(MAX \equiv \max_{\mathcal{S}'_{1:K}}\log\det(\Sigma(\hat{x}_{1:K}|\mathcal{S}'_{1:K})))$.}

\item \textit{\emph{Time complexity of Algorithm \ref{alg:general}:} 
Algorithm \ref{alg:general} has time complexity of order $O(n^{2.4}K\sum_{k=1}^K r_k^2)$.}
\end{enumerate}
\end{mytheorem}

Theorem \ref{th:alg_performance} extends to continuous and discrete time-variant systems as follows:

\begin{mycorollary}\label{cor:cont}
\textit{Consider the time-variant version of~\eqref{eq:dynamics}: 
\begin{equation}\label{eq:cont_dynamics}
\dot{x}(t)= A(t)x(t)+B(t)u(t)+F(t)w(t), t \geq t_0.
\end{equation}
\begin{enumerate}[1)]
\item Part 1 of Theorem \ref{th:alg_performance} holds.
\item Part 2 of Theorem \ref{th:alg_performance} holds if the time complexity for computing each transition matrix $\Phi(t_{k+1}, t_{k})$ \cite{Chen:1998:LST:521603}, where $k \in [K-1]$, is $O(n^3)$.\footnote{The matrices $\Phi(t_{k+1}, t_{k})$, where $k \in [K-1]$, are used in the computation of $\Sigma(\hat{x}_{1:K})$ (cf.~proof of Theorem \ref{th:alg_performance} in Appendix \ref{app:proof_of_theorem}).}
\end{enumerate}}
\end{mycorollary}

\begin{mycorollary}\label{cor:discrete}
\textit{Consider the discrete time version of \eqref{eq:cont_dynamics}: 
\begin{equation}\label{eq:disc_dynamics}
x{[k+1]}=A_kx[k]+B_ku[k]+F_kw[k], k \geq k_0.
\end{equation}
Similarly, consider the discrete time counterparts of the sensor model \eqref{eq:sensors}, Assumption \ref{ass:indep}, and the sensor scheduling model \eqref{eq:measurements}.
\begin{enumerate}[1)]
\item Part 1 of Theorem \ref{th:alg_performance} holds.
\item Part 2 of Theorem \ref{th:alg_performance} holds if $A_k$ in \eqref{eq:disc_dynamics} is full rank for all $k \in [K]$.
\end{enumerate}}
\end{mycorollary}

We follow-up with several remarks on Theorem \ref{th:alg_performance}:

\begin{myremark}\label{rem:approx_quality}
\emph{(Approximation quality of Algorithm \ref{alg:general})} Theorem \ref{th:alg_performance}  quantifies the worst-case performance of Algorithm \ref{alg:general} across all values of Problem 1's parameters.  The reason is that the right-hand side of \eqref{ineq:opt_guar} is constant.  In particular, \eqref{ineq:opt_guar} guarantees that for any instance of Problem 1, the distance of the approximate value $\log\det(\Sigma(\hat{x}_{1:K}|\mathcal{S}_{1:K}))$ from $OPT$ is at most $1/2$ the distance of the worst (maximum) value $MAX$ from $OPT$.  
In addition, this approximation factor is near to the optimal approximation factor $1/e \cong .38$ one can achieve in the worst-case for Problem 1 in polynomial time \cite{vondrak2010submodularity}; the reason is twofold: first, as we comment in the next paragraph, we prove that Problem 1 involves the minimization of a non-increasing and supermodular function \cite{Nemhauser:1988:ICO:42805}, and second, as we proved in Section \ref{sec:comp_complx}, Problem 1 is in the worst-case equivalent to the minimal controllability problem introduced in \cite{olshevsky2014minimal}, which cannot be approximated in polynomial time with a better factor than the $1/e$ \cite{Feige:1998:TLN:285055.285059}.  
\end{myremark}

\begin{myremark}\label{rem:sub}
\emph{$($Supermodularity of $\log\det(\Sigma(\hat{x}_{1:K})))$} In the proof of Theorem \ref{th:alg_performance} (Appendix \ref{app:proof_of_theorem}), we show that $\log\det(\Sigma(\hat{x}_{1:K}))$ is a non-increasing and supermodular function with respect to the sequence of selected sensors.
{Specifically, the proof of \eqref{ineq:opt_guar} follows by combining these two results and results on the maximization of submodular functions over matroid constraints \cite{fisher1978analysis} ---we present these three derivations in Appendices \ref{sec:mon}, \ref{sec:set_sub}, and \ref{sec:proof_of_main_theorem}, respectively.}
\end{myremark}

We continue with our third remark on Theorem \ref{th:alg_performance}:
\begin{myremark}\label{rem:time_c}
\emph{(Time complexity of Algorithm \ref{alg:general})} Algorithm \ref{alg:general}'s time complexity is broken down into two parts: the first part is the number of evaluations of $\log\det(\Sigma(\hat{x}_{1:K}))$ required by the algorithm, and the second part is the time complexity of each such evaluation.  In particular, Algorithm \ref{alg:general} requires at most $r_k^2$ evaluations of $\log\det(\Sigma(\hat{x}_{1:K}))$ at each $t_k$.    
Therefore, Algorithm \ref{alg:general} achieves a time complexity that is only linear in $K$ with respect to the total number of evaluations of $\log\det(\Sigma(\hat{x}_{1:K}))$.  The reason is that $\sum_{k=1}^K r_k^2 \leq \max_{k\in [K]}(r_k^2)K$.  In addition, for $w(t)$ zero mean and white Gaussian  ---as commonly assumed in the literature of sensor scheduling--- the time complexity of each such evaluation is at most linear in $K$:  the reason is that this $w(t)$ agrees with Assumption \ref{ass:indep}, in which case we prove that the time complexity of each evaluation of $\log\det(\Sigma(\hat{x}_{1:K}))$ is $O(n^{2.4}K)$ (linear in~$K$).\footnote{We can also speed up Algorithm \ref{alg:general} by implementing in Algorithm~\ref{alg:greedy_alg} the method of lazy evaluations \cite{minoux1978accelerated}: this method avoids in Step 2 of Algorithm~\ref{alg:greedy_alg} the computation of $\rho_i(\mathcal{S}^{t-1})$ for unnecessary choices of $i$.}
\end{myremark}

\begin{myremark}
\emph{$($Sparsity of $\Sigma(\hat{x}_{1:K}))$}
We state the three properties of $\log\det(\Sigma(\hat{x}_{1:K}))$ we prove to obtain the time complexity for Algorithm~\ref{alg:general}. The 
first two properties were mentioned in Remark \ref{rem:sub}: the monotonicity and supermodularity of $\log\det(\Sigma(\hat{x}_{1:K}))$.  These two properties are responsible for that Algorithm \ref{alg:general} requires at most $r_k^2$ evaluations at each $t_k$.  The third property, which follows, is responsible for the low time complexity for each evaluation of $\log\det(\Sigma(\hat{x}_{1:K}))$: 
\begin{itemize}
\item $\Sigma(\hat{x}_{1:K})$ is the sum of two $nK\times nK$ sparse matrices: the first matrix is block diagonal, and the second one is block tri-diagonal.  As a result, given that both of these matrices are known, each evaluation of $\log\det(\Sigma(\hat{x}_{1:K}))$ has time complexity $O(n^{2.4}K)$, linear in~$K$ (using the results in \cite{molinari2008determinants} ---cf.~Theorem 2 therein).
\end{itemize}
We show in Appendix \ref{sec:proof_of_main_theorem} that after we include at each evaluation step of $\log\det(\Sigma(\hat{x}_{1:K}))$ the complexity to compute the two sparse matrices in $\Sigma(\hat{x}_{1:K})$, the total time complexity of Algorithm \ref{alg:general} is as given in Theorem~\ref{th:alg_performance}.
\end{myremark}

\begin{algorithm}[tr]
\caption{Single step greedy algorithm (subroutine in Algorithm \ref{alg:general}).}
\begin{algorithmic}
\REQUIRE Current iteration $k$ (corresponds to $t_k$), selected sensor sets $(\mathcal{S}_1, \mathcal{S}_2,\ldots, \mathcal{S}_{k-1})$ up to the current iteration, scheduling constraint $r_k$, estimation error function $\log\det(\Sigma(\hat{x}_{1:K}|\mathcal{S}_{1:K})): \mathcal{S}_k \subseteq [m], k \in [K] \mapsto \mathbb{R}$
\ENSURE Sensor set $\mathcal{S}_k$ that approximates the solution to Problem 1 at $t_k$
\vspace{0.25mm}
\STATE $\mathcal{S}^0\leftarrow\emptyset$, $\mathcal{X}^0\leftarrow [m]$, and $t\leftarrow 1$
\STATE \textbf{Iteration t:}
\begin{enumerate}[1.]
\item If $\mathcal{X}^{t-1}=\emptyset$, \textbf{return} $\mathcal{S}^{t-1}$
\item Select $i(t)\in \mathcal{X}^{t-1}$ for which $\rho_{i(t)}(\mathcal{S}^{t-1})=\max_{i \in \mathcal{X}^{t-1}}\rho_i(\mathcal{S}^{t-1})$, with ties settled arbitrarily, where:
\begin{eqnarray*}
\rho_i(\mathcal{S}^{t-1}) &\equiv &\log\det(\Sigma(\hat{x}_{1:K}|\mathcal{S}_{1:k-1}, \mathcal{S}^{t-1}))-\\
&&\quad\log\det(\Sigma(\hat{x}_{1:K}|\mathcal{S}_{1:k-1}, \mathcal{S}^{t-1}\cup\{i\}))
\end{eqnarray*}
and $\mathcal{S}_{1:k-1}\equiv (\mathcal{S}_1, \mathcal{S}_2,\ldots, \mathcal{S}_{k-1})$
\item[{3.a.}] If $|\mathcal{S}^{t-1}\cup \{i(t)\}| > r_k$, $\mathcal{X}^{t-1}\leftarrow \mathcal{X}^{t-1}\setminus\{i(t)\}$, and go to Step 1
\item[{3.b.}] If $|\mathcal{S}^{t-1}\cup \{i(t)\}| \leq r_k$, $\mathcal{S}^t\leftarrow \mathcal{S}^{t-1}\cup \{i(t)\}$ and $\mathcal{X}^{t}\leftarrow \mathcal{X}^{t-1}\setminus\{i(t)\}$
\item[{4.}] $t\leftarrow t+1$ and continue
\end{enumerate}
\end{algorithmic} \label{alg:greedy_alg}
\end{algorithm}

Our final remark on Theorem \ref{th:alg_performance} follows:

\begin{myremark} \label{rem:comparison}
\emph{(Comparison of Algorithm \ref{alg:general}'s time complexity to that of existing scheduling algorithms)}
We do the comparison for two cases: batch state estimation, and Kalman filtering.  In particular, we show that the time complexity of our algorithm is lower than that of existing sensor scheduling algorithms for batch state estimation, and of the similar order, or lower, of existing algorithms for Kalman filtering.  

\paragraph*{Comparison with algorithms for batch state estimation} In \cite{Roy2016369}, Problem 1 is considered, and a semi-definite programming (SDP) algorithm is proposed; its time complexity is of the order $O(\max_{k \in [K]}(r_k)K(nK)^{3.5}+(\max_{k\in [K]}(r_k^2)K^2(nK)^{2.5})$~\cite{nemirovski2004interior}.  Clearly, this time complexity is higher than that of Algorithm \ref{alg:general}, whose complexity is $O(\max_{k\in [K]}(r_k)^2K^2n^{2.4})$.  In addition, the algorithm presented in \cite{Roy2016369} provides no worst-case approximation guarantees \eqref{ineq:opt_guar}, in contrast to Algorithm \ref{alg:general} that provides~\eqref{ineq:opt_guar}.

\paragraph*{Comparison with algorithms for Kalman filtering}
We do the comparison in two steps: first, we consider algorithms based on the maximization of submodular functions, and second, algorithms based on convex relaxation techniques or the alternating direction method of multipliers~(ADMM):
\begin{itemize}
\item \emph{Algorithms based on the maximization of submodular functions:}  In \cite{jawaid2015submodularity}, an algorithm is provided that is valid for a restricted class of linear systems: its time complexity is $O(\max_{k\in [K]}(r_k)m n^2K+n^{2.4}K)$.  This time complexity is of similar order to that of Algorithm \ref{alg:general}, whose complexity is of the order $O(\max_{k\in [K]}(r_k)^2Kn^{2.4}K)$, since $\max_{k\in [K]}(r_k)<m$.  Specifically, we observe in Algorithm \ref{alg:general}'s time complexity the additional multiplicative factor $K$ (linear in $K)$; this difference emanates from that Algorithm \ref{alg:general} offers a near-optimal guarantee over the whole time horizon $(t_1, t_2, \ldots, t_K)$ whereas the algorithm in \cite{jawaid2015submodularity} offers a near-optimal guarantee only for the last time step $t_K$. In addition, Algorithm \ref{alg:general} holds for any linear continuous time-invariant system (no restrictions are necessary),  in contrast to the algorithm in~\cite{jawaid2015submodularity}, and it holds for any discrete time-variant systems where $A_k$ in \eqref{eq:disc_dynamics} is full rank; the latter assumption is one of the four restrictive conditions in~\cite{jawaid2015submodularity} (Theorem 13).

\item  \emph{Algorithms based on convex relaxation techniques or ADMM:} In \cite{joshi2009sensor}, the authors assume a single sensor ($r_k=1$ across $t_k$), and their objective is to achieve a minimal estimation error by minimizing the number of times this sensor will be used over the horizon $t_1, t_2, \ldots, t_K$.  The time complexity of the proposed algorithm is $O(n^{2.5}K^2+n^{3.5}K)$.  This time complexity is higher than that of Algorithm \ref{alg:general}, whose complexity for $r_k=1$ is of the order $O(n^{2.4}K^2)$.  In \cite{liu2014optimal}, the authors employ ADMM techniques to solve a periodic sensor scheduling problem.  They consider a zero mean and white Gaussian $w(t)$.  The time complexity of the proposed algorithm is $O((nK)^3+(\max_{k\in[K]}(r_k)K)n^2K^2+max(r_k)^2nK^3)$.  This time complexity is of similar order to that of Algorithm \ref{alg:general}, whose complexity in this case is $O(\max_{k\in[K]}(r_k^2)n^{2.4}K^2)$, since $\max_{k\in[K]}(r_k)\leq K$; in particular, for $K>n^{0.4}\max_{k\in[K]}(r_k)$, Algorithm \ref{alg:general} has lower time complexity.\footnote{More algorithms exist in the literature, that also use convex relaxation \cite{weimer2008relaxation} or randomization techniques \cite{gupta2006stochastic}, and have similar time complexity to Algorithm \ref{alg:general}.  They achieve this complexity using additional approximation methods: e.g., they optimize instead an upper bound to the involved estimation error metric.}
\end{itemize}  
\end{myremark}

With the above remarks we conclude: Algorithm~\ref{alg:general} enjoys both the estimation accuracy of the batch state scheduling algorithms and the low time complexity of the Kalman filtering scheduling algorithms, since:
\begin{itemize}
\item Algorithm \ref{alg:general} offers a near-optimal worst-case approximation guarantee for the batch state estimation error.  This estimation error is only approximated by the Kalman filtering sensor scheduling algorithms: the reason is that they aim instead to minimize the sum of each of the estimation errors for $x(t_k)$ (across $t_k$).  However, this sum only upper bounds the batch state estimation error.

\item Algorithm \ref{alg:general} has time complexity lower than the state of the art batch estimation algorithms, and at the same time, lower than, or similar to, the time complexity of the corresponding Kalman filtering algorithms. 
\end{itemize} 
In addition: Algorithm \ref{alg:general}'s approximation guarantee holds for any linear system (continuous or discrete time).  Moreover, Algorithm \ref{alg:general}'s time complexity guarantee holds for any continuous time system, and for discrete time systems where $A_k$ in \eqref{eq:disc_dynamics} is full rank across~$k$.

The proof of Theorem \ref{th:alg_performance} can be found in Appendix \ref{app:proof_of_theorem}.  

\subsection{Limits on Sensor Scheduling for Minimum Variance Batch State Estimation}\label{sec:limitations}

In this section,  we derive two trade-offs between three important parameters of our sensor scheduling problem:\footnote{We recall from Section \ref{sec:Prelim} that the objective of Problem 1 is related to $\mathbb{E}(\|x_{1:K}-\hat{x}_{1:K}\|_2^2)$ in that when $\log\det(\Sigma(\hat{x}_{1:K}))$ is minimized the probability that the estimation error $\|x_{1:K}-\hat{x}_{1:K}\|_2^2$ is small is maximized.}
\begin{itemize}
\item the number of measurements times $(t_1, t_2, \ldots, t_K)$
\item the number $r_k$ of sensors that can be used at each $t_k$
\item the value of the estimation error $\mathbb{E}(\|x_{1:K}-\hat{x}_{1:K}\|_2^2)$.
\end{itemize}
The first of the two trade-offs is captured in the next theorem:

\begin{mytheorem}\label{th:perfomance_lim}
\textit{Let $\sigma^{(-1)}_w \equiv \max_{i \in [nK]}[\mathbb{C}(x_{1:K})^{-1}]_{ii}$ and $\sigma^{(-1)}_v \equiv \|\mathbb{C}(v_{1:K})^{-1}\|_2$.  Also, let $C_{1:K}$ be the block diagonal matrix where each of its $K$ diagonal elements is equal to $C$, where $C$ is the matrix $[C_1^\top, C_2^\top, \ldots, C_m^\top]^\top$. For the variance of the error of the minimum variance estimator~$\hat{x}_{1:K}$:
\begin{equation}\label{eq:lim_0}
\begin{split}
&\mathbb{E}(\|x_{1:K}-\hat{x}_{1:K}\|_2^2)\geq \\
&\qquad\qquad\quad\frac{n}{\sigma^{(-1)}_v\max_{k\in[K]}(r_k)\|C_{1:K}\|_2^2+ \sigma^{(-1)}_w/K}.
\end{split}
\end{equation} }
\end{mytheorem}
\vspace{2mm}
The lower bound in \eqref{eq:lim_0} decreases as the number of used sensors for scheduling $r_k$ increases or the number measurement times $K$ increases, and increases as the system's size increases.  Since these qualitative relationships were expected, the importance of this theorem lies on the quantification of these relationships (that also includes the dependence on the noise parameters $\sigma^{(-1)}_w$ and $\sigma^{(-1)}_v$): for example, \eqref{eq:lim_0} decreases only inversely proportional with the number of sensors for scheduling; that is, increasing the number $r_k$ so to reduce the variance of the error of the minimum variance estimator is ineffective, a fundamental limit.  In addition, this bound increases linearly with the system's size; this is another limit for large-scale systems.

{Similar results are proved in \cite{5705695} for the steady state error covariance of scalar systems in the case that the number of sensors goes to infinity.  In more detail, the authors in \cite{5705695} account for different types of multi-access schemes, as well as, for fading channels between the sensors and the fusion centre that combines the sensor measurements.}

The next corollary presents our last trade-off:
\begin{mycorollary}\label{cor:sensors_tradeoff}
\textit{Consider that the desired value for $\mathbb{E}(\|x_{1:K}-\hat{x}_{1:K}\|_2^2)$ is $\alpha$.  Any set of scheduled sensors at $t_1, t_2, \ldots, t_K$ that achieves this error satisfies:
\begin{equation}\label{eq:limitation_sensors}
\max_{k\in[K]}(r_k)\geq \frac{n/\alpha-\sigma^{(-1)}_w/K}{  \sigma^{(-1)}_v\|C_{1:K}\|_2^2}.
\end{equation}} 
\end{mycorollary}
\vspace{1mm}

Eq.~\eqref{eq:limitation_sensors} implies that the number of sensors used for scheduling at each $t_k$ increases as the error of the minimum variance estimator or the number of measurements times $K$ decreases.  More importantly, it quantifies that this number increases linearly with the system's size for fixed error variance.  This is again a fundamental limit, meaningful for large-scale systems.

\section{Future Work}\label{sec:conc}

We work on extending the results of this paper to largely unknown systems, under the presence of non-linear measurements.  The first of these extensions allows systems whose evolution is captured by, e.g., Gaussian processes or random networks (the former example is a widely used assumption for motion models; cf.~\cite{anderson2015batch} and references therein).  The second of these extensions allows complex measurement environments, such as camera-sensor environments, that can enable the application of our results in domains such as robotics and the automotive~sector. 


\section{Acknowledgements}

Vasileios Tzoumas thanks Nikolay A. Atanasov for bringing into his attention paper \cite{anderson2015batch}.


\appendices

\section{Closed formula for the error covariance of $\hat{x}_{1:K}$}\label{app:estimator}

We compute the error covariance of $\hat{x}_{1:K}$:  Denote as $S_{1:K}$ the block diagonal matrix with diagonal elements the sensor selection matrices $S(t_1), S(t_2), \ldots, S(t_K)$.  Moreover, denote as $C$ the matrix $[C_1^\top, C_2^\top, \ldots, C_m^\top]^\top$. Finally, denote  
$y_{1:K} \equiv (y(t_1)^\top, y(t_2)^\top, \ldots, y(t_k)^\top)^\top$,
$w_{1:K} \equiv (w(t_1)^\top, w(t_2)^\top,\ldots, w(t_k)^\top)^\top$, and
$v_{1:K}$ $\equiv (v(t_1)^\top,v(t_2)^\top, \ldots, v(t_k)^\top)^\top$, where $v(t_k) \equiv (v_1(t_k)^\top,$ $v_2(t_k)^\top, \ldots, v_m(t_k)^\top)^\top]$.  Then, from \eqref{eq:dynamics}, \eqref{eq:sensors} and \eqref{eq:measurements}:
\begin{eqnarray}\label{eq:initial_to_output}
y_{1:K} &=& O_{1:K} x_{1:K}+S_{1:K}v_{1:K},
\end{eqnarray}
where $O_{1:K}$ is the $\sum_{k=1}^Kr_k\times nK$ block diagonal matrix with diagonal elements the matrices $S(t_1)C, S(t_2)C, \ldots, S(t_K)C$. $\hat{x}_{1:K}$ has the error covariance $\Sigma(\hat{x}_{1:K})=\mathbb{E}(( x_{1:K}-\hat{x}_{1:K})(x_{1:K}-\hat{x}_{1:K})^\top)$ \cite{kailath2000linear}:
\begin{equation}
\Sigma(\hat{x}_{1:K})=\mathbb{C}(x_{1:K})-\mathbb{C}(x_{1:K})O_{1:K}^\top \Xi
O_{1:K}\mathbb{C}(x_{1:K}), \label{eq:Sigma_z}
\end{equation}
where $\Xi\equiv(O_{1:K}\mathbb{C}(x_{1:K})O_{1:K}^\top +S_{1:K}\mathbb{C}(v_{1:K})S_{1:K}^\top)^{-1}$.

We simplify \eqref{eq:Sigma_z} in the following lemma:

\begin{mylemma}\label{lem:closed_formula}
\textit{The error covariance of $\hat{x}_{1:K}$ has the equivalent form:
\begin{equation}\label{eq:closed_formula}
\Sigma(\hat{x}_{1:K})
=\left(\sum_{k=1}^K\sum_{i=1}^m s_i(t_k) U^{(ki)}+\mathbb{C}(x_{1:K})^{-1}\right)^{-1},
\end{equation}
where $s_i(t_k)$ is a zero-one function, equal to $1$ if and only if sensor $i$ is used at $t_k$, and $U^{(ki)}$ is the block diagonal matrix $C_{1:K}^\top I^{(ki)}\mathbb{C}(v_{1:K})^{-1}I^{(ki)}C_{1:K}$; $C_{1:K}$ is the block diagonal matrix where each of its $K$ diagonal elements is equal to $C$, and $I^{(ki)}$ is the block diagonal matrix with $mK$ diagonal elements such that: the $((k-1)m+i)$-th element is the $d_i \times d_i $ identity matrix $I$, and the rest of the elements are equal to zero.}
\end{mylemma}
\begin{proof} 
Let $\Phi\equiv(S_{1:K}\mathbb{C}(v_{1:K})S_{1:K}^\top)^{-1}$; then,
\begin{equation}
\Sigma(\hat{x}_{1:K})
=(O_{1:K}^\top\Phi O_{1:K}+\mathbb{C}(x_{1:K})^{-1})^{-1}\label{eq:lse_aux_222},
\end{equation}
where we deduce \eqref{eq:lse_aux_222} from \eqref{eq:Sigma_z} using the Woodbury matrix identity (Corollary 2.8.8 in~\cite{bernstein2009matrix}).  In addition, because $S_{1:K}$ is block diagonal, and $\mathbb{C}(v_{1:K})$ is block diagonal as well (per Assumption \ref{ass:indep}), $\Phi=S_{1:K}\mathbb{C}(v_{1:K})^{-1}S_{1:K}^\top$.  Moreover, $O_{1:K}=S_{1:K}C_{1:K}$.  Then, 
\begin{equation}
\Sigma(\hat{x}_{1:K})
=(C_{1:K}^\top D\mathbb{C}(v_{1:K})^{-1} DC_{1:K}+\mathbb{C}(x_{1:K})^{-1})^{-1},
\end{equation}
where $D=S_{1:K}^\top S_{1:K}$.  Now, since $D$ and $\mathbb{C}(v_{1:K})^{-1}$ are block diagonal, $\mathbb{C}(v_{1:K})^{-1} D=D\mathbb{C}(v_{1:K})^{-1}$.  Furthermore, the definition of $S_{1:K}$ implies $D^2=D$.  As a result,
\begin{equation}\label{eq:aux_444}
\Sigma(\hat{x}_{1:K})
=(C_{1:K}^\top D\mathbb{C}(v_{1:K})^{-1}C_{1:K}+\mathbb{C}(x_{1:K})^{-1})^{-1}.
\end{equation}

Moreover, we observe: $D=\sum_{k=1}^K\sum_{i=1}^m s_i(t_k)I^{(ki)}$.  In addition, for any $k \in [K]$ and $i \in [m]$: \[
C_{1:K}^\top I^{(ki)}\mathbb{C}(v_{1:K})^{-1}C_{1:K}=C_{1:K}^\top I^{(ki)}\mathbb{C}(v_{1:K})^{-1}I^{(ki)}C_{1:K},
\]
which is the reverse step we used before for $D$.  Using the last observation in \eqref{eq:aux_444} the proof is complete.
\end{proof}

\section{Proof of Theorem \ref{th:np}}\label{app:np}

\begin{proof}
We present an instance of Problem 1 that is equivalent to the NP-hard minimal observability problem introduced in \cite{olshevsky2014minimal, sergio2015minimal}, that is defined as follows:

\paragraph*{Definition (Minimal Observability Problem)}
\textit{Consider the linear time-invariant system:
\begin{align}\label{eq:min_obs_syst}
\begin{split}
\dot{x}(t)&= A x(t), \\
y_i(t) &= s_{i} e_i^\top x(t), i \in [n] 
\end{split}
\end{align}
where $x(t) \in \mathbb{R}^n$, $e_i$ is the vector with the $i$-th entry equal to $1$ and the rest equal to $0$, and $s_{i}$ is either zero or one; the \emph{minimal observability problem} follows:
\begin{equation}
\begin{aligned}
& {\text{select}} 
 \; \qquad s_1, s_2, \ldots, s_n \\
&\text{such that } \hspace{2mm}s_1+ s_2+ \ldots+ s_n \leq r,\\
& \qquad\qquad \hspace{1.5mm} \eqref{eq:min_obs_syst} \text{ is observable},
\end{aligned}
\end{equation}
where $r \leq n$.}

The minimal observability problem is NP-hard when $A$ is chosen as in the proof of Theorem 1 of \cite{olshevsky2014minimal}.  We denote this $A$ as $A_{NP-h}$.

Problem 1 is equivalent to the NP-hard minimal observability problem for the following instance: $K=1$, $m=n$, $w(t)=0$ and $v_i(t)=0$, $C_i=e_i^\top$, $r_1=r$, and $A(t)= A_{NP-h}$.  This observation concludes the proof.
\end{proof}

\section{Proof of Theorem \ref{th:alg_performance}} \label{app:proof_of_theorem}

We prove Theorem \ref{th:alg_performance} in three steps: we first show that $\log\det(\Sigma(\hat{x}_{1:K}))$ is a non-increasing function in the choice of the~sensors; we then show that $\log\det(\Sigma(\hat{x}_{1:K}))$ is a supermodular function in the choice of the~sensors; finally, we prove Theorem \ref{th:alg_performance} by combining the aforementioned two results and results on the maximization of submodular functions over matroid constraints \cite{fisher1978analysis}.

\paragraph*{Notation} We recall that any collection $(x_1,x_2, \ldots, x_k)$ is denoted as $x_{1:k}$ ($k \in \mathbb{N}$).

\subsection{Monotonicity in Sensor Scheduling for Minimum Variance Batch State Estimation}\label{sec:mon}

We first provide two notations, and then the definition of non-increasing and non-decreasing set functions.  Afterwards, we present the main result of this subsection. 
  
\paragraph*{Notation} Given $K$ disjoint finite sets $\mathcal{X}_1,\mathcal{X}_2, \ldots,\mathcal{X}_K$ and $A_{i}, B_{i} \in \mathcal{X}_{i}$, we write $A_{1:K} \preceq B_{1:K}$ to denote that for all $i \in [K]$, $A_i\subseteq B_i$ ($A_i$ is a subset of $B_i$).  Moreover, we denote that $A_{i} \in \mathcal{X}_{i}$ for all $i \in [K]$ as $A_{1:K} \in \mathcal{X}_{1:K}$.

\begin{mydef}
\textit{Consider $K$ disjoint finite sets $\mathcal{X}_1,\mathcal{X}_2, \ldots,\mathcal{X}_K$.  A function $h: \mathcal{X}_{1:K} \mapsto \mathbb{R}$ is \emph{non-decreasing} if and only if for all $A,B \in \mathcal{X}_{1:K}$ such that $A \preceq B$,
$
h(A)\leq h(B);
$
$h: \mathcal{X}_{1:K}\mapsto \mathbb{R}$ is \emph{non-increasing} if $-h$ is non-decreasing.}
\end{mydef}

The main result of this subsection follows: 

\begin{myproposition}\label{prop:entropy_monoton}
\textit{For any finite $K \in\mathbb{N}$, consider  $K$ distinct copies of $[m]$, denoted as $\mathcal{M}_1, \mathcal{M}_2, \ldots, \mathcal{M}_K$.  The estimation error metric
$
\log\det(\Sigma(\hat{x}_{1:K}|\mathcal{S}_{1:K})): \mathcal{M}_{1:K} \mapsto \mathbb{R}
$
is a non-increasing function in the choice of the~sensors  $\mathcal{S}_{1:K}$.}
\end{myproposition}
\begin{proof} 
Consider $\mathcal{S}_{1:K} \preceq \mathcal{S}_{1:K}'$.  From \eqref{eq:closed_formula} (Lemma \ref{lem:closed_formula} in Appendix \ref{app:estimator}) and Theorem 8.4.9 in \cite{bernstein2009matrix}, $\Sigma(\hat{x}_{1:K}|\mathcal{S}_{1:K}')\preceq\Sigma(\hat{x}_{1:K}|\mathcal{S}_{1:K})$, since $U^{(ki)} \succeq 0$ and $\mathbb{C}(x_{1:K})\succ 0$ (i.e., $\mathbb{C}(x_{1:K})^{-1}\succ 0$).  As a result, $\log\det(\Sigma(\hat{x}_{1:K}|\mathcal{S}_{1:K}'))\preceq \log\det(\Sigma(\hat{x}_{1:K}|\mathcal{S}_{1:K}))$, and the proof is complete.
\end{proof}

We next show that $\log\det(\Sigma(\hat{x}_{1:K}|\mathcal{S}_{1:K}))$ is a supermodular function with respect to the selected sensors $\mathcal{S}_{1:K}$.
 
\subsection{Submodularity in Sensor Scheduling for Minimum Variance Batch State Estimation}\label{sec:set_sub}

We first provide a notation, and then the definition of submodular and supermodular set functions.  Afterwards, we present the main result of this subsection. 

\paragraph*{Notation} Given $K$ disjoint finite sets $\mathcal{X}_1,\mathcal{X}_2,\ldots,\mathcal{X}_K$ and $A_{1:K}, B_{1:K} \in \mathcal{X}_{1:K}$, we write $A_{1:K} \uplus B_{1:K}$ to denote that for all $i \in [K]$, $A_i\cup B_i$ ($A_i$ union $B_i$). 

\begin{mydef}\label{def:seq_sub}
\textit{Consider $K$ disjoint finite sets $\mathcal{X}_1,\mathcal{X}_2, \ldots,$ $\mathcal{X}_K$. A function $h: \mathcal{X}_{1:K}\mapsto \mathbb{R}$ is \emph{submodular} if and only if for all $A,B,C \in \mathcal{X}_{1:K}$ such that $A \preceq B$,
$
h(A\uplus C)-h(A)\geq h(B\uplus C)-h(B);
$
$h: \mathcal{X}_{1:K}\mapsto \mathbb{R}$ is \emph{supermodular} if $-h$ is submodular.}
\end{mydef}
 
According to Definition \ref{def:seq_sub}, set submodularity is a diminishing returns property: a function $h: \mathcal{X}_{1:K}\mapsto \mathbb{R}$ is set submodular if and only if for all $C \in \mathcal{X}_{1:K}$, the function $h_C: \mathcal{X}_{1:K}\mapsto \mathbb{R}$ defined for all $A \in \mathcal{X}_{1:K}$ as $h_C(A)\equiv h(A\uplus C)-h(A)$ is non-increasing.

The main result of this subsection follows:

\begin{myproposition}
\label{prop:seq_sub}
\textit{For any finite $K \in\mathbb{N}$, consider $K$ distinct copies of $[m]$, denoted as $\mathcal{M}_1, \mathcal{M}_2, \ldots, \mathcal{M}_K$; the estimation error metric
$
\log\det(\Sigma(\hat{x}_{1:K}|\mathcal{S}_{1:K})): \mathcal{M}_{1:K} \mapsto \mathbb{R}
$
is a set supermodular function in the choice of the~sensors $\mathcal{S}_{1:K}$.}
\end{myproposition}
\begin{proof}
Denote~$\sum_{k=1}^K\sum_{i=1}^m s_i(t_k) U^{(ki)}$~in~\eqref{eq:closed_formula}~(Lemma \ref{lem:closed_formula} in Appendix \ref{app:estimator}) as $U(\mathcal{S}_{1:K})$, $-\log\det(\Sigma(\hat{x}_{1:K}|\mathcal{S}_{1:K}))$ as $h(\mathcal{S}_{1:K})$, and $h(\mathcal{S}_{1:K}\cup \{a\})-h(\mathcal{S}_{1:K})$ as $h_a(\mathcal{S}_{1:K})$, for any $a \in \bigcup_{k=1}^{K}\mathcal{M}_k$.

To prove that $\log\det(\Sigma(\hat{x}_{1:K}|\mathcal{S}_{1:K}))$ is supermodular, it suffices to prove that $h(\mathcal{S}_{1:K})$ is submodular.  
In particular, $h$ is submodular if and only if $h_a$ is a non-increasing, for any $a \in \bigcup_{k=1}^{K}\mathcal{M}_k$.  To this end, we follow the proof of Theorem 6 in~\cite{summers2014submodularity}: first, observe that:
\begin{align*}
h_a(\mathcal{S}_{1:K})&= \log\det(U(\mathcal{S}_{1:K}\cup \{a\})+\mathbb{C}(x_{1:K})^{-1})-\\
&\qquad\qquad\qquad\qquad\hspace{-2mm}\log\det(U(\mathcal{S}_{1:K})+\mathbb{C}(x_{1:K})^{-1})\\
&=\log\det(U(\mathcal{S}_{1:K})+U(\{a\})+\mathbb{C}(x_{1:K})^{-1})-\\
&\qquad\qquad\qquad\qquad\hspace{-2mm}\log\det(U(\mathcal{S}_{1:K})+\mathbb{C}(x_{1:K})^{-1}).
\end{align*}
For $\mathcal{S}_{1:K}\preceq\mathcal{S}_{1:K}'$ and $t \in [0,1]$, define~$O(t)\equiv$ $ \mathbb{C}(x_{1:K})^{-1}+U(\mathcal{S}_{1:K})+t(U(\mathcal{S}_{1:K}')-U(\mathcal{S}_{1:K}))$ and 
$g(t)\equiv \log\det\left(O(t)+U(\{a\})\right)-\log\det\left(O(t)\right).$
Then, $g(0)=h_a(\mathcal{S}_{1:K})$ and $g(1)=h_a(\mathcal{S}_{1:K}')$.  Moreover, since: 
\[\frac{d\log\det(O(t))}{dt}=\text{tr}\left(O(t)^{-1} \frac{dO(t)}{dt}\right)
\]
(eq.~(43) in~\cite{petersen2008matrix}),
\[\dot{g}(t)= \text{tr}\left[((O(t)+U(\{a\}))^{-1}-O(t)^{-1})L\right],
\]
where $L\equiv U(\mathcal{S}_{1:K}')-U(\mathcal{S}_{1:K})$. From Proposition 8.5.5 in~\cite{bernstein2009matrix}, 
$
(O(t)+U(\{a\}))^{-1} \preceq O(t)^{-1}
$
($O(t)\succ 0$, since $\mathbb{C}(x_{1:K})^{-1} \succ 0$, $U(\mathcal{S}_{1:K}) \succeq 0$, and $U(\mathcal{S}_{1:K}')\succeq U(\mathcal{S}_{1:K})$).  Moreover, from Corollary 8.3.6 in~\cite{bernstein2009matrix},
all the eigenvalues of $((O(t)+U(\{a\}))^{-1}-O(t)^{-1})L$ are non-positive.
As a result, $\dot{g}(t)\leq 0$, and
$
h_a(\mathcal{S}_{1:K}')=g(1)=g(0)+\int_{0}^{1}\dot{g}(t)dt\leq g(0)=h_a(\mathcal{S}_{1:K}).
$
Therefore, $h_a$ is non-increasing, and the proof is complete.
\end{proof}

Proposition \ref{prop:seq_sub} implies that as we increase at each $t_k$ the number of sensors used, the marginal improvement we get on the estimation error of $x_{1:K}$ diminishes.

We are now ready for the proof of Theorem \ref{th:alg_performance}.
 
\subsection{Proof of Theorem \ref{th:alg_performance}}\label{sec:proof_of_main_theorem}

We first provide the definition of a matroid, and then continue with the main proof:
\begin{mydef}\label{def:indep_matroid}
\textit{Consider a finite set $\mathcal{X}$ and a collection $\mathcal{C}$ of subsets of $\mathcal{X}$. $(\mathcal{X}, \mathcal{C})$ is:
\begin{itemize}
\item an \emph{independent system} if and only if:
\begin{itemize}
\item $\emptyset \in \mathcal{C}$, where $\emptyset$ denotes the empty set
\item for all $X' \subseteq X \subseteq \mathcal{X}$, if $X \in \mathcal{C}$, $X' \in \mathcal{C}$.
\end{itemize}
\item a \emph{matroid} if and only if in addition to the previous two properties:
\begin{itemize}
\item for all $ X', X \in \mathcal{C}$ where $|X'| < |X|$, there exists $x \notin X'$ and $x \in X$ such that $X' \cup \{x\} \in \mathcal{C}$.
\end{itemize}
\end{itemize}}
\end{mydef}

\begin{proof}[\textit{of Part 1 of Theorem \ref{th:alg_performance}}]
We use the next result from the literature of maximization of submodular functions over matroid constraints:
\begin{mylemma}[Ref.~\cite{fisher1978analysis}]\label{lem:sub_guarantees}
\textit{Consider $K$ independence systems $\{(\mathcal{X}_k,\mathcal{C}_k)\}_{k \in [K]}$, each the intersection of at most $P$ matroids,
\footnote{Any independence system can be expressed as the intersection of a finite number of matroids \cite{Nemhauser:1988:ICO:42805}.} 
and a submodular and non-decreasing function $h: \mathcal{X}_{1:K}\mapsto \mathbb{R}$. There exist a polynomial time greedy algorithm that returns an (approximate) solution $\mathcal{S}_{1:K}$ to: 
\begin{equation}\label{pr:general}
\begin{aligned}
& \underset{\mathcal{S}_{1:K} \preceq \mathcal{X}_{1:K}}{\text{maximize}} 
 \; \quad h(\mathcal{S}_{1:K}) \\
&\hspace{0mm}\text{subject to} \quad\hspace{1mm} \mathcal{S}_k\cap \mathcal{X}_k \in \mathcal{C}_k, k \in [K],
\end{aligned}
\end{equation}
that satisfies:  
\begin{equation}\label{ineq:opt_guar_seq_sub}
\frac{h(\mathcal{O})-h(\mathcal{S}_{1:K})}{h(\mathcal{O})-h(\emptyset)}\leq \frac{P}{1+P},
\end{equation}
where $\mathcal{O}$ is an (optimal) solution to \eqref{pr:general}.}
\end{mylemma}
In particular, we prove:
\begin{mylemma}\label{lemma:as_Problem_1.6}
\textit{Problem 1 is an instance of \eqref{pr:general} with $P=1$.}
\end{mylemma}
\begin{proof}
We identify the instance of $\{\mathcal{X}_k,\mathcal{C}_k\}_{k\in [K]}$ and $h$, respectively, that translate \eqref{pr:general} to Problem 1:  

Given $K$ distinct copies of $[m]$, denoted as $\mathcal{M}_1, \mathcal{M}_2, \ldots, \mathcal{M}_K$, first consider $\mathcal{X}_k=\mathcal{M}_k$ and $\mathcal{C}_k=$ $\{\mathcal{S}|\mathcal{S}\subseteq \mathcal{M}_k, |\mathcal{S}|\leq r_k\}$: $(\mathcal{X}_k, \mathcal{C}_k)$ satisfies the first two points in part 1 of Definition \ref{def:indep_matroid}, and as a result is an independent system.  Moreover, by its definition, $\mathcal{S}_k\cap \mathcal{X}_k \in \mathcal{C}_k$ if and only if $|\mathcal{S}_k|\leq r_k$.

Second, for all $\mathcal{S}_{1:K} \preceq \mathcal{X}_{1:K}$, consider: \[h(\mathcal{S}_{1:K})=-\log\det(\Sigma(\hat{x}_{1:K}|\mathcal{S}_{1:K})).
\]From Propositions \ref{prop:entropy_monoton} and \ref{prop:seq_sub}, $h(\mathcal{S}_{1:K})$ is set submodular and non-decreasing. 
In addition to Lemma \ref{lemma:as_Problem_1.6}, the independence system $(\mathcal{X}_k, \mathcal{C}_k)$, where $\mathcal{X}_k=\mathcal{M}_k$ and $\mathcal{C}_k=\{\mathcal{S}|\mathcal{S}\subseteq \mathcal{M}_k, |\mathcal{S}|\leq r\}$, satisfies also the point in part 2 of Definition \ref{def:indep_matroid}; thereby, it is also a matroid and as a result $P$, as defined in Lemma \ref{lem:sub_guarantees}, is equal to $1$.  
\end{proof}
This observation, along with Lemmas \ref{lem:sub_guarantees} and \ref{lemma:as_Problem_1.6} complete the proof of \eqref{ineq:opt_guar}, since the adaptation  to Problem 1 of the greedy algorithm in \cite{fisher1978analysis} (Theorem 4.1) results to Algorithm~\ref{alg:general}.
\end{proof}

\begin{proof}[\emph{of Part 2 of Theorem \ref{th:alg_performance}}]
In Lemma \ref{lem:closed_formula} in Appendix \ref{app:estimator} we prove that $\Sigma(\hat{x}_{1:K})$ is the sum of two matrices: the first matrix is a block diagonal matrix, and the second one is the inverse of the  covariance of $x_{1:K}$, $\mathbb{C}(x_{1:K})$.  The block diagonal matrix is computed in $O(n^{2.4}K)$ time.  Moreover, by extending the result in \cite{anderson2015batch} (Theorem 1), we get that $\mathbb{C}(x_{1:K})^{-1}$ is a block tri-diagonal matrix, that is described by the $(K-1)$ transition matrices $\Phi(t_{k+1}, t_{k})$ \cite{Chen:1998:LST:521603}, where $k \in [K-1]$,  and $K$ identity matrices.  For continuous time systems, the time complexity to compute all the block elements in $\mathbb{C}(x_{1:K})^{-1}$ is $O(n^3K)$ \cite{moler2003nineteen}; for discrete time systems, it is $O(n^{2.4}K)$ \cite{Chen:1998:LST:521603}.  This computation of $\mathbb{C}(x_{1:K})^{-1}$ is made only once.  Finally, from Theorem 2 in \cite{molinari2008determinants}, we can now compute the $\det(\Sigma(\hat{x}_{1:K}))$ in $O(n^{2.4}K)$ time, since $\Sigma(\hat{x}_{1:K})$ is block tri-diagonal.  Therefore, the overall time complexity of Algorithm \ref{alg:general} is:  $O(n^3K)+O(2n^{2.4}K\sum_{k=1}^{K}r^2_k)=O(n^{2.4}K\sum_{k=1}^{K}r^2_k)$ for $K$ large, since $\mathbb{C}(x_{1:K})^{-1}$ is computed only once, and Algorithm \ref{alg:general} requests at most $\sum_{k=1}^{K}r^2_k$ evaluations of $\Sigma(\hat{x}_{1:K})$.

The proof is complete.~
\end{proof}

\section{Proof of Theorem \ref{th:perfomance_lim}} \label{app:proof_of_perf_lim}

\begin{proof}
Since the arithmetic mean of a finite set of positive numbers is at least as large as their harmonic mean,
\begin{equation*}
\text{tr}(\Sigma(\hat{x}_{1:K}))\geq \frac{(nK)^2}{\text{tr}\left(\sum_{k=1}^K\sum_{i=1}^m s_i(t_k) U^{(ki)}+\mathbb{C}(x_{1:K})^{-1}\right)}.
\end{equation*}

In the denominator, we have for the second term: $\text{tr}(\mathbb{C}(x_{1:K})^{-1})\leq nK \sigma^{(-1)}_w$; 
and for the first term:
\begin{equation*}
\text{tr}\left(\sum_{k=1}^K\sum_{i=1}^m s_i(t_k) U^{(ki)}\right)=\sum_{k=1}^K\sum_{i=1}^m s_i(t_k)\text{tr}(U^{(ki)}),
\end{equation*}
where $\text{tr}(U^{(ki)})\leq nK \|C_{1:K}\|_2^2\|\mathbb{C}(v_{1:K})^{-1}\|_2$, since $\|I^{(ki)}\|_2=1$.  Then, we get the upper bound:
\begin{align*}
\text{tr}\left(\sum_{k=1}^K\sum_{i=1}^m s_i(t_k) U^{(ki)}\right)&\leq nK\sigma^{(-1)}_v\|C_{1:K}\|_2^2\sum_{k=1}^Kr_k\\
&\leq nK^2\sigma^{(-1)}_v\|C_{1:K}\|_2^2\max_{k\in[K]}(r_k).
\end{align*}
Overall: (we denote $\max_{k\in[K]}(r_k)$ as $r_M$)
\begin{align*}
\text{tr}(\Sigma(\hat{x}_{1:K}))&\geq \frac{(nK)^2}{nK^2\sigma^{(-1)}_vr_M\|C_{1:K}\|_2^2+ nK \sigma^{(-1)}_w}\\
&=\frac{n}{\sigma^{(-1)}_vr_M\|C_{1:K}\|_2^2+ \sigma^{(-1)}_w/K},
\end{align*}
and the proof is complete.
\end{proof}

\bibliographystyle{IEEEtran}
\bibliography{references}

\end{document}